\documentclass[12pt,leqno,draft]{article} 
\usepackage{amsmath,amssymb,amsthm,amsfonts,bm}
\usepackage{enumerate,color}
\makeatletter
\def\@cite#1#2{[{{\bfseries #1}\if@tempswa , #2\fi}]}
\renewcommand{\section}{%
\@startsection{section}{1}{\z@}
{0.5truecm plus -1ex minus -.2ex}%
{1.0ex plus .2ex}{\bfseries\large}}
\def\@seccntformat#1{\csname the#1\endcsname.\ }
\makeatother

\numberwithin{equation}{section} 
\newtheorem{thm}{Theorem}[section]

\newtheorem{lem}[thm]{Lemma}

\theoremstyle{definition}

\newtheorem*{prth1.1}{Proof of Theorem 1.1}
\newtheorem*{prth1.2}{Proof of Theorem 1.2}
\newtheorem*{prth1.3}{Proof of Theorem 1.3}

\newcommand{\ep}{\varepsilon}

\newcommand{\tmax}{T_{\rm max}}
\newcommand{\lp}[2]{\|#2\|_{L^{#1}(\Omega)}}
\newcommand{\into}{\int_\Omega}

\begin{document}

\footnote[0]
    {2010{\it Mathematics Subject Classification}\/. 
    Primary: 35K45; Secondary: 92C17; 35Q35.
    }
\begin{center}
    \Large{{\bf Global existence and asymptotic behavior of classical solutions for a 3D two-species Keller--Segel-Stokes system with competitive kinetics 
               }}
\end{center}
\vspace{5pt}
\begin{center}
    Xinru Cao \\
    \vspace{2pt}
  Institut f\"ur Mathematik, Universit\"at Paderborn\\ 
  Warburger Str.100, 33098 Paderborn, Germany\\
    \vspace{12pt}
    Shunsuke Kurima\\
    \vspace{2pt}
    Department of Mathematics, 
    Tokyo University of Science\\
    1-3, Kagurazaka, Shinjuku-ku, Tokyo 162-8601, Japan\\
    \vspace{12pt}
     Masaaki Mizukami\footnote{Corresponding author}\\
    \vspace{2pt}
    Department of Mathematics, 
    Tokyo University of Science\\
    1-3, Kagurazaka, Shinjuku-ku, Tokyo 162-8601, Japan\\
    {\tt masaaki.mizukami.math@gmail.com}\\
\end{center}
\begin{center}    
    \small \today
\end{center}

\vspace{2pt}
\newenvironment{summary}
{\vspace{.5\baselineskip}\begin{list}{}{%
     \setlength{\baselineskip}{0.85\baselineskip}
     \setlength{\topsep}{0pt}
     \setlength{\leftmargin}{12mm}
     \setlength{\rightmargin}{12mm}
     \setlength{\listparindent}{0mm}
     \setlength{\itemindent}{\listparindent}
     \setlength{\parsep}{0pt} 
     \item\relax}}{\end{list}\vspace{.5\baselineskip}}
\begin{summary}
{\footnotesize {\bf Abstract.}
This paper deals with the two-species Keller--Segel-Stokes system 
with competitive kinetics 
\begin{equation*}
     \begin{cases}
         (n_1)_t + u\cdot\nabla n_1 =
          \Delta n_1 - \chi_1\nabla\cdot(n_1\nabla c) 
          + \mu_1n_1(1- n_1 - a_1n_2),
         &x\in \Omega,\ t>0,
 \\[2mm]
         (n_2)_t + u\cdot\nabla n_2 =
          \Delta n_2 - \chi_2\nabla\cdot(n_2\nabla c) 
          + \mu_2n_2(1- a_2n_1 - n_2),
         &x\in \Omega,\ t>0,
 \\[2mm]
         c_t + u\cdot\nabla c = 
         \Delta c - c + \alpha n_1 +\beta n_2,
          &x \in \Omega,\ t>0,
 \\[2mm]
        u_t  
          = \Delta u + \nabla P 
            + (\gamma n_1 + \delta n_2)\nabla\phi, 
            \quad \nabla\cdot u = 0, 
          &x \in \Omega,\ t>0
     \end{cases} 
 \end{equation*} 
under homogeneous Neumann boundary conditions in a bounded domain 
$\Omega \subset \mathbb{R}^3$ with smooth boundary. 
Many mathematicians study chemotaxis-fluid systems and 
two-species chemotaxis systems with 
competitive kinetics. 
However, there are not many results on coupled 
two-species chemotaxis-fluid systems 
which have difficulties of the chemotaxis effect, 
the competitive kinetics and the fluid influence. 
Recently, in the two-species chemotaxis-Stokes system, 
where $-c+\alpha n_1+\beta n_2$ is replaced with 
$-(\alpha n_1+\beta n_2)c$ in the above system, 
global existence and asymptotic behavior 
of classical solutions were obtained 
in the 3-dimensional case 
under the condition that 
$\mu_1,\mu_2$ are sufficiently large (\cite{CKM}). 
Nevertheless, 
the above system has not been studied yet; 
we cannot apply the same argument as in the 
previous works 
because of lacking the $L^\infty$-information of $c$. 
The main purpose of this paper is 
to obtain global existence and stabilization 
of classical solutions to the above system in the 3-dimensional case 
under the largeness conditions for $\mu_1,\mu_2$. 
}
\end{summary}
\vspace{10pt}

\newpage

\section{Introduction and results}

%
%
It is found in many biological experiments that cells 
have the ability to adapt their migration in response 
to a chemical signal in their neighbourhood; 
especially, cells move towards higher concentration 
of a chemical substance which is produced by 
themselves. 
This mechanism plays a central role in mathematical 
biology, and has many applications in its variants 
and extensions (see \cite{B-B-T-W}). 

Let $n$ denote the density of the cells and $c$ present 
the concentration of the chemical signal. 
The mathematical model describing the 
above mechanism which was first proposed 
by Keller--Segel \cite{K-S} reads as
\begin{align*}
n_t = \Delta n - \nabla\cdot (n \nabla c), \quad 
c_t = \Delta c - c +n. 
\end{align*}
In the above system 
which is called Keller--Segel system, 
it is known that the size of initial data determine 
whether classical solutions of 
the above system exist globally or not 
(\cite{Xinru_higher,Horstmann-Wang,Nagai-Senba-Yoshida,win_aggregationvs}). 
In the 2-dimensional setting 
global classical solutions of the above system exist  
when the mass of an initial data $n_0$ is 
sufficiently small (\cite{Nagai-Senba-Yoshida}). 
On the other hand, 
there exist initial data such that 
the mass of the initial data is large enough and 
a solution blows up in finite time 
in a 2-dimensional bounded domain (\cite{Horstmann-Wang}). 
In the higher dimensional case 
it is known that, if 
an initial data $(n_0,c_0)$ is 
sufficiently 
small with respect to suitable Lebesgue norm, 
then global existence and boundedness of 
classical solutions hold 
(\cite{Xinru_higher,win_aggregationvs}). 
Then we expect that 
existence of blow-up solutions hold 
under some largeness condition for initial data 
also in the higher-dimensional case; 
however, Winkler \cite{win_aggregationvs} showed that 
for all $m>0$ there exist initial data $n_0$ 
such that $\lp{1}{n_0}=m$ and a solution blows up 
in finite time in the case that the domain is a ball 
in $\mathbb{R}^N$ with $N\ge 3$. 

As we mentioned above, we can see that 
whether solutions 
of the Keller--Segel system blow up or not 
depends on initial data. 
On the other hand, it is known that the logistic term 
can prevent solutions from blowing up even 
though a initial data is large enough 
(\cite{lankeit_evsmoothness,OTYM,Winkler_2010_logistic}). 
Osaki--Tsujikawa--Yagi--Mimura \cite{OTYM} 
obtained that, in the 2-dimensional case, 
the chemotaxis system with logistic source 
\begin{align*}
n_t=\Delta n - \chi\nabla \cdot (n\nabla c) + r n - \mu n^2, \quad
c_t= \Delta c -c +n, 
\end{align*} 
where $\chi,r,\mu>0$, 
possesses a unique global solution for all 
$\chi,r,\mu>0$ and all initial data. 
In the higher dimensional case, 
Winkler \cite{Winkler_2010_logistic,W2014} 
established existence of 
global classical solutions 
under the condition that $\mu>0$ is sufficiently large. 
Moreover, asymptotic behavior of the solutions 
was obtained: 
$n(\cdot,t)\to \frac{r}{\mu}$,  
$c(\cdot,t)\to \frac{r}{\mu}$
in $L^\infty(\Omega)$
as $t\to\infty$ (\cite{he_zheng}).  
Recently, Lankeit \cite{lankeit_evsmoothness} 
obtained global existence 
of weak solutions to 
the chemotaxis system with logistic source 
for all $r\in \mathbb{R}$ and any $\mu>0$ 
and the eventual smoothness of 
the solutions was derived if $r\in \mathbb{R}$ is small. 
More related works 
which deal with the Keller--Segel system and 
the chemotaxis system with logistic source 
can be found in \cite[Section 3]{B-B-T-W}. 

The chemotaxis system with logistic source is 
a {\it single-species} case, 
and has a classical solution which converges to 
the constant steady state 
$(\frac{r}{\mu},\frac{r}{\mu})$. 
On the other hand, 
the {\it two-species} chemotaxis-competition system 
\begin{gather*}
         (n_1)_t  =
          \Delta n_1 - \chi_1\nabla\cdot(n_1\nabla c) 
          + \mu_1n_1(1- n_1 - a_1n_2),
 \\
         (n_2)_t  =
          \Delta n_2 - \chi_2\nabla\cdot(n_2\nabla c) 
          + \mu_2n_2(1- a_2n_1 - n_2),
 \\
         c_t = \Delta c - c + \alpha n_1 + \beta n_2,
\end{gather*}
where $\chi_1,\chi_2,\mu_1,\mu_2,a_1,a_2,\alpha,\beta>0$, 
which describes the evolution of 
two competing species which react on a single chemoattractant, 
has different dynamics depending on $a_1$ and $a_2$ 
(\cite{Bai-Winkler_2016,Black-Lankeit-Mizukami_01,
Lin-Mu-Wang,Mizukami_DCDSB,Mizukami_improve,
stinner_tello_winkler,Tello_Winkler_2012}). 
In the 2-dimensional case Bai--Winkler 
\cite{Bai-Winkler_2016} obtained 
global existence of the above system for all parameters.  
In the higher-dimensional case global existence and boundedness 
in the above system were established 
in \cite{Lin-Mu-Wang,Mizukami_DCDSB}. 
Moreover, it was shown that the solutions of 
the above system have the same asymptotic behavior 
as the solutions of Lotka--Volterra competition model: 
$
 n_1(\cdot,t)\to \frac{1-a_1}{1-a_1a_2}$, 
$
 n_2(\cdot,t)\to \frac{1-a_2}{1-a_1a_2}$, 
$
 c(\cdot,t)\to \frac{\alpha(1-a_1)+\beta(1-a_2)}{1-a_1a_2}
$ 
in $L^\infty(\Omega)$ as $t\to\infty$ in the case that $a_1,a_2\in (0,1)$, and 
$
 n_1(\cdot,t)\to 0$, 
$
 n_2(\cdot,t)\to 1$, 
$
 c(\cdot,t)\to \beta
$ 
in $L^\infty(\Omega)$ as $t\to\infty$ in the case that 
$a_1\ge 1> a_2>0$ (\cite{Bai-Winkler_2016,Mizukami_DCDSB,Mizukami_improve}). 
More related works can be found in 
\cite{Black-Lankeit-Mizukami_01,Lin-Mu-Wang,Mizukami_PPEpro,stinner_tello_winkler,Tello_Winkler_2012}. 

Recently, 
the chemotaxis-fluid system
\begin{gather*}
         n_t + u\cdot\nabla n =
          \Delta n - \chi\nabla\cdot(n\nabla c) 
          + r n - \mu n^2,
 \\
         c_t + u\cdot\nabla c 
         = \Delta c + g(n,c),
 \\
        u_t  + \kappa (u\cdot\nabla) u 
          = \Delta u + \nabla P 
            + n\nabla\phi, 
            \quad \nabla\cdot u = 0, 
\end{gather*}
where $\chi>0$, $r,\mu\ge0$, $\kappa=0$ 
(the Stokes case) or 
$\kappa=1$ (the Navier--Stokes case),  
which is the chemotaxis system with the fluid influence 
according to the {\it Navier--Stokes} equation, 
was intensively studied 
(\cite{Li-Xiao,Tao-Winkler_2015_3D,
Tao-Winkler_2016_2D,W-2012}). 
In the case that $g(n,c)=-nc$ and $r=\mu=0$ 
(chemotaxis-Navier--Stokes system) 
Winkler \cite{W-2012} first overcame 
the difficulties of the chemotactic effect 
and the fluid influence, 
and showed existence of global classical solutions 
in the 2-dimensional case and 
global existence of weak solutions 
to the above system with $\kappa=0$ in the 3-dimensional setting; 
even though in the 3-dimensional chemotaxis-Stokes case, 
classical solutions were not found. 
In the case that $g(n,c)=-c+n$ and $\kappa=0$ 
(Keller--Segel-Stokes system) 
Li--Xiao \cite{Li-Xiao} showed global 
existence and boundedness 
under the smallness condition for 
the mass of initial data 
only in the 2-dimensional case. 
On the other hand, also in the chemotaxis-fluid system, 
the logistic source can be helpful 
for obtaining classical bounded solutions; 
global classical bounded solutions were established 
in the 2-dimensional Keller--Segel-Navier--Stokes 
system (\cite{Tao-Winkler_2016_2D}) 
and in the 3-dimensional Keller--Segel-Stokes system 
under the condition that $\mu$ is large enough 
(\cite{Tao-Winkler_2015_3D}). 
For more related works we refer to 
\cite[Section 4]{B-B-T-W}. 

As we discussed previously, 
the chemotaxis-competition system and 
the chemotaxis-fluid system were intensively studied. 
However, there are not many results on a coupled 
two-species chemotaxis-fluid system 
which have difficulties of the chemotaxis effect, 
the competitive kinetics and the fluid influence. 
Recently, the problem which is a combination of 
the chemotaxis-Navier--Stokes system 
and the chemotaxis-competition system was 
studied in the 
2-dimensional case and the 3-dimensional case 
(\cite{CKM,HKMY_1}); 
in the 2-dimensional case global existence 
and asymptotic behavior 
of classical bounded solutions to the 
two-species chemotaxis-Navier--Stokes system 
were established (\cite{HKMY_1}), and 
in the 3-dimensional case 
existence and stabilization of 
global classical bounded solutions to 
the two-species chemotaxis-Stokes system 
hold under the condition that 
$\mu_1,\mu_2$ are sufficiently large (\cite{CKM}). 
However, the two-species Keller--Segel-Stokes system 
with competitive kinetics 
 \begin{equation}\label{P}
     \begin{cases}
         (n_1)_t + u\cdot\nabla n_1 =
          \Delta n_1 - \chi_1\nabla\cdot(n_1\nabla c) 
          + \mu_1n_1(1- n_1 - a_1n_2),
         &x\in \Omega,\ t>0,
 \\[0.5mm]
         (n_2)_t + u\cdot\nabla n_2 =
          \Delta n_2 - \chi_2\nabla\cdot(n_2\nabla c) 
          + \mu_2n_2(1- a_2n_1 - n_2),
         &x\in \Omega,\ t>0,
 \\[0.5mm]
         c_t + u\cdot\nabla c = \Delta c - c + \alpha n_1 + \beta n_2,
          &x \in \Omega,\ t>0,
 \\[0.5mm]
        u_t   
          = \Delta u + \nabla P 
            + (\gamma n_1 + \delta n_2)\nabla\phi, 
            \quad \nabla\cdot u = 0, 
          &x \in \Omega,\ t>0, 
 \\[1mm]
        \partial_\nu n_1 
        = \partial_\nu n_2 = \partial_\nu c = 0, \quad 
        u = 0, 
        &x \in \partial\Omega,\ t>0, 
 \\[0.5mm]
        n_i(x,0)=n_{i,0}(x),\ 
        c(x,0)=c_0(x),\ u(x,0)=u_0(x), 
        &x \in \Omega,\ i=1,2,
     \end{cases}
 \end{equation}
\noindent
where $\Omega$ is a bounded domain, 
$\chi_1, \chi_2, a_1, a_2 \ge 0$ and 
$\mu_1, \mu_2, \alpha, \beta, \gamma, \delta > 0$ are 
constants, 
$n_{1,0}, n_{2,0}, c_0, u_0, \phi$ 
are known functions, 
has not been studied yet; 
we cannot apply the same argument as in \cite{CKM} 
because of lacking the $L^\infty$-information 
and only having $L^1$-estimate for $c$. 
in this case. 
Indeed, we cannot pick $r\in (1,3)$ such that  
\begin{align*}
\frac{5-\frac{3}{p+1}}{1-\frac{3}{(p+1)\theta}}\cdot
(p+1)\cdot \frac{\frac{1}{r}-\frac{3}{(p+1)\theta}}{\frac{1}{2}-\frac{3}{r}}
<2
\end{align*}
holds at \cite[(3.9)]{CKM} (in the case that 
we use $\lp{(p+1)\theta'}{\nabla c}\le 
C\lp{p+1}{\Delta c}^a\lp{1}{c}^{1-a}$ instead of 
\cite[(3.8)]{CKM}). 
The purpose of the present paper is to 
obtain global existence and asymptotic behavior 
in \eqref{P} in a 3-dimensional domain.  
In order to attain this purpose, we assume 
throughout this paper that the known functions 
$n_{1,0}, n_{2,0}, c_0, u_0, \phi$ satisfy
 \begin{align}\label{condi;ini1}
   &0 < n_{1,0}, n_{2,0} 
   \in C^0(\overline{\Omega}), 
 \quad 
   0 < c_0 \in W^{1,q}(\Omega), 
 \quad 
   u_0 \in D(A^{\vartheta}), 
  \\\label{condi;ini2}
   &\phi \in C^{1+\eta}(\overline{\Omega}) 
 \end{align}
for some $q > 3$, 
$\vartheta \in \left(\frac{3}{4}, 1\right)$, 
$\eta > 0$ and $A$ is 
the Stokes operator (see \cite{Sohr_NS}). 
Then the main results read as follows. 
The first theorem gives global existence and 
boundedness in \eqref{P}. 
%

\begin{thm}\label{maintheorem1}
Let $\Omega \subset \mathbb{R}^3$ be a 
bounded domain with smooth boundary and let $\chi_1, 
\chi_2>0$, $a_1, a_2 \ge 0$, 
$\mu_1, \mu_2, \alpha, \beta, \gamma, \delta > 0$. 
Assume that \eqref{condi;ini1} and \eqref{condi;ini2} 
are satisfied. 
Then there exists a constant $\xi_0 > 0$ such that 
whenever 
\[
\chi := \max\{\chi_1, \chi_2\} \quad \mbox{and}\quad 
\mu := \min\{\mu_1, \mu_2\}
\] 
satisfy $\frac{\chi}{\mu} < \xi_0$, 
there is a classical solution $(n_1, n_2, c, u, P)$ of the problem \eqref{P}
such that 
\begin{align*}
&n_1, n_2 \in C^0(\overline{\Omega}\times[0, \infty)) 
\cap C^{2, 1}(\overline{\Omega}\times(0, \infty)), 
\\
&c \in C^0(\overline{\Omega}\times[0, \infty)) 
\cap C^{2, 1}(\overline{\Omega}\times(0, \infty)) 
\cap L^{\infty}_{{\rm loc}}([0, \infty); W^{1, q}(\Omega)), 
\\
&u \in C^0(\overline{\Omega}\times[0, \infty)) 
\cap C^{2, 1}(\overline{\Omega}\times(0, \infty))
\cap L^\infty_{\rm loc}([0,\infty);D(A^{\vartheta})), 
\\  
&P \in C^{1, 0}(\overline{\Omega} \times (0, \infty)).
\end{align*} 
Also, the solution is unique in the sense that 
it allows up to 
addition of spatially constants to the pressure $P$. 
Moreover, 
the above solution is bounded in the following sense\/{\rm :}  
   \[
   \sup_{t>0} (\|n_1(\cdot, t)\|_{L^{\infty}(\Omega)} 
                     + \|n_2(\cdot, t)\|_{L^{\infty}(\Omega)} 
                     + \|c(\cdot, t)\|_{W^{1, q}(\Omega)} 
                     + \|u(\cdot, t)\|_{D(A^\vartheta)}) 
   < \infty.
   \]
\end{thm}

The second theorem asserts 
asymptotic behavior of solutions to \eqref{P}.


 \begin{thm}\label{maintheorem2}
Assume that the assumption of 
Theorem \ref{maintheorem1} is satisfied. 
Then the following properties hold\/{\rm :} 
   \begin{enumerate}
   \item[{{\rm (i)}}] In the case that $a_1, a_2 \in (0, 1)$, 
   under the condition that there exists $\delta_1>0$ 
   such that 
   \begin{align*}
     &4\delta_1 - (1+\delta_1)^2 a_1a_2 > 0 
     \quad \mbox{and}
   \\
     &\frac{\chi_1^2(1-a_1)}{4a_1\mu_1(1-a_1a_2)}
     +
     \frac{\delta_1\chi_2^2(1-a_2)}{4a_2\mu_2(1-a_1a_2)}
     <   \frac{4\delta_1 - (1+\delta)^2a_1a_2}
    {a_1\alpha^2 \delta + a_2\beta^2 -a_1a_2\alpha\beta (1+\delta_1)},
   \end{align*}
       the solution of the problem \eqref{P} converges to 
       a constant stationary solution of \eqref{P} as follows\/{\rm :}
       $$
       n_1(\cdot, t) \to N_1,\quad n_2(\cdot, t) \to N_2,\quad 
       c(\cdot, t) \to C^*,\quad 
       u(\cdot, t) \to 0 \quad \mbox{in}\ L^{\infty}(\Omega)\quad \mbox{as}\ t \to \infty, 
       $$
where 
       \[
       N_1 := \frac{1 - a_1}{1 - a_1a_2},\quad N_2 := \frac{1 - a_2}{1 - a_1a_2} 
       ,\quad C^* := \alpha N_1 + \beta N_2.
       \]
   \item[{{\rm (ii)}}] 
   In the case that $a_1 \geq 1 > a_2$, 
   under the condition that there exist 
   $\delta_1' > 0$ and $a_1' \in [1, a_1]$ such that 
  \begin{align*}
    &4\delta_1'-a_1'a_2(1+\delta_1')^2 > 0
    \quad \mbox{and}
  \\
    &\mu_2 > 
    \frac{\chi_2^2\delta_1'(\alpha^2 a_1'\delta_1' + \beta^2 a_2 
    - \alpha\beta a_1'a_2(1+\delta_1'))}{4a_2(4\delta_1'-a_1'a_2(1+\delta_1')^2)}, 
  \end{align*}
      the solution of the problem \eqref{P} converges to 
      a constant stationary solution of \eqref{P} as follows\/{\rm :}
      $$
      n_1(\cdot, t) \to 0,\quad n_2(\cdot, t) \to 1,\quad c(\cdot, t) \to \beta,\quad 
       u(\cdot, t) \to 0 \quad \mbox{in}\ L^{\infty}(\Omega)\quad \mbox{as}\ t \to \infty.
      $$
   \end{enumerate}
\end{thm}

The strategy for the proof of Theorem \ref{maintheorem1} 
is to confirm the $L^p$-estimates for $n_1$ and $n_2$ 
with $p>\frac{3}{2}$. 
By using the differential inequality we can obtain  
\begin{align*}
 \int_\Omega n_1^p +\int_\Omega n_2^p 
 \le C + C\int_{s_0}^t e^{-(p+1)(t-s)}\int_\Omega |\Delta c|^{p+1}  
 -C\int_{s_0}^{t}e^{-(p+1)(t-s)}
 \Bigl(\int_{\Omega}n_1^{p+1} + \int_{\Omega}n_2^{p+1} \Bigr)
\end{align*}
with some $C>0$ and $s_0>0$. 
We will control $\int_{s_0}^t e^{-(p+1)(t-s)}\int_\Omega |\Delta c|^{p+1}$ 
by applying the variation of 
the maximal Sobolev regularity (Lemma \ref{pre2}) 
for the third equation in \eqref{P}.  
Then we can obtain the $L^p$-estimates for $n_1$ and $n_2$. 
Here the keys of this strategy are the $L^2$-estimate 
for $\nabla c$ (Lemma \ref{lem;grad;c}) 
and the maximal Sobolev regularity for the 
Stokes equation (Lemma \ref{po}); 
these enable us to overcome difficulties of 
applying an argument 
similar to that in \cite{CKM}. 
On the other hand, 
the strategy for the proof of 
Theorem \ref{maintheorem2} is 
to confirm the following inequality: 
\begin{align}\label{st;ineq}
  \int_0^\infty\int_\Omega (n_1-N_1)^2 
  + \int_0^\infty\int_\Omega (n_2-N_2)^2 
  + \int_0^\infty\int_\Omega (c-C^*)^2   
  \le C
\end{align}
with some $C>0$, where $(N_1,N_2,C^*,0)$ 
is a constant stationary solution of \eqref{P}. 
In order to obtain this estimate 
we will find a nonnegative function $E$ satisfying 
\begin{align*}
 \frac{d}{dt}E(t)\le 
 -\ep \int_\Omega \left[(n_1-N_1)^2+(n_2-N_2)^2+(c-C^*)^2\right]
\end{align*} 
with some $\ep>0$. 
The above inequality and the positivity of $E(t)$ enable 
us to attain the desired estimate \eqref{st;ineq}. 

The plan of this paper is as follows. 
In Section 2 we collect basic facts which will be used later. 
Section 3 is devoted to proving 
global existence and boundedness (Theorem \ref{maintheorem1}). 
In Section 4 
we show asymptotic stability (Theorem \ref{maintheorem2}). 


\section{Preliminaries}

In this section we will give  
some results which 
will be used later. 
We can prove the following lemma which gives 
local existence of classical solutions to \eqref{P} by a straightforward adaptation of the reasoning in \cite[Lemma 2.1]{W-2012}.


 \begin{lem}\label{pre1}
 Let $\Omega \subset \mathbb{R}^3$ be a bounded domain with smooth boundary.  
Assume that \eqref{condi;ini1} and \eqref{condi;ini2} are satisfied. 
Then there exists $\tmax \in (0, \infty]$ such that 
the problem \eqref{P} possesses a classical solution 
$(n_1, n_2, c, u, P)$ satisfying   
   \begin{align*}
&n_1, n_2 \in C^0(\overline{\Omega}\times[0, \tmax)) 
\cap C^{2, 1}(\overline{\Omega}\times(0, \tmax)), \\ 
&c \in C^0(\overline{\Omega}\times[0, \tmax)) 
\cap C^{2, 1}(\overline{\Omega}\times(0, \tmax)) 
\cap L^{\infty}_{{\rm loc}}([0, \tmax); W^{1, q}(\Omega)), \\ 
&u \in C^0(\overline{\Omega}\times[0, \tmax)) 
\cap C^{2, 1}(\overline{\Omega}\times(0, \tmax)) 
\cap L^\infty_{{\rm loc}}([0, \tmax); D(A^{\vartheta})), 
\\  
&P \in C^{1, 0}(\overline{\Omega} \times (0, \tmax)), 
\\
&n_1,n_2>0,\quad c>0\quad \mbox{in}\ 
\Omega\times (0,\tmax).
  \end{align*}
Also, the above solution is unique up to addition of spatially constants 
to the pressure $P$. 
Moreover, either $\tmax = \infty$ or 
\[
    \|n_1(\cdot, t)\|_{L^{\infty}(\Omega)} + \|n_2(\cdot, t)\|_{L^{\infty}(\Omega)} 
    + \|c(\cdot, t)\|_{W^{1, q}(\Omega)} 
    + \|A^{\vartheta}u(\cdot, t)\|_{L^2(\Omega)} 
   \to \infty \quad \mbox{as}\ t\nearrow \tmax.
\]
 \end{lem}

Given $s_0 \in (0, \tmax)$, 
we can derive from the regularity properties that 
  \[c(\cdot, s_0) \in C^2(\overline{\Omega})
\quad 
\mbox{with}\quad 
\partial_{\nu}c(\cdot, s_0) = 0 
\quad \mbox{on}\ \partial\Omega. 
\]
In particular, there exists 
$M = M(s_0)> 0$ such that 
   $$ 
   \|c(\cdot, s_0)\|_{W^{2, \infty}(\Omega)} \leq M
   $$
(see e.g., \cite{YCJZ-2015}). 

The following two lemmas provided as variations of 
the maximal Sobolev regularity 
hold keys for  
global existence and boundedness of solutions to \eqref{P}. 

\begin{lem}\label{pre2} 
Let $s_0 \in (0, \tmax)$. 
Then for all $p>1$ 
there exists a constant $C=C(p)>0$ 
such that 
\begin{align*}
  \int_{s_0}^{t}\int_{\Omega}e^{ps}|\Delta c|^p
  &\leq C\int_{s_0}^{t}\int_{\Omega}e^{ps}
       |c + \alpha n_1 + \beta n_2 - u\cdot \nabla c|^p 
\\&\quad\,
       + C e^{ps_0}
       (\|c(\cdot, s_0)\|_{L^p(\Omega)}^p 
       + \|\Delta c(\cdot, s_0)\|_{L^p(\Omega)}^p)
\end{align*}
holds for all $t\in (s_0,\tmax)$. 
\end{lem}
\begin{proof}
The proof is similar to that of 
\cite[Lemma 2.2]{CKM}. 
Let $s_0\in (0,\tmax)$ and $t\in (s_0,\tmax)$. 
By using the transformation 
$\widetilde{c}(\cdot, s)=e^s c(\cdot, s)$, $s\in (s_0,t)$, 
and the maximal Sobolev regularity 
\cite[Theorem 3.1]{Hieber-Press_Maximal-Sobolev} for $\widetilde{c}$ 
we obtain this lemma. 
\end{proof}
\begin{lem}\label{po}
Let $s_0 \in (0,\tmax)$. Then 
for all $p > 1$ there exists a constant 
$C=C(p, s_0) > 0$ such that 
 \begin{align*}
     \int_{s_0}^{t}e^{ps}\int_{\Omega}|Au|^{p} 
     \leq C
     \left(
     \int_{s_0}^{t} e^{ps}\into |u|^p 
     + \int_{s_0}^{t} e^{ps}\into n_1^p 
     + \int_{s_0}^{t} e^{ps}\into n_2^p\right)  
     + C
     \end{align*}
for all $t \in (s_0, \tmax)$.
\end{lem}
\begin{proof}
Letting $s_0\in (0,\tmax)$ and $t\in (s_0,\tmax)$ and 
putting $\widetilde{u}(\cdot, s):=e^{s}u(\cdot, s)$, 
$s\in (s_0,t)$, 
we obtain from the forth equation in \eqref{P} that 
$$
\widetilde{u}_{s} = \Delta \widetilde{u} + \widetilde{u} 
                       + e^{s}(\gamma n_1 + \delta n_2)\nabla \phi + e^{s}\nabla P, 
$$
which derives  
\[
\widetilde{u}_{s} + A\widetilde{u} 
= {\cal P}[\widetilde{u} + e^{s}(\gamma n_1 + \delta n_2)\nabla \phi], 
\] 
where ${\cal P}$ denotes the Helmholtz projection 
mapping $L^2(\Omega)$ onto its subspace 
$L^2_\sigma(\Omega)$ of all solenoidal 
vector field. 
Thus we derive from \cite[Theorem 2.7]{GS-1990} that 
there exist positive constants 
$C_1=C_1(p,s_0)$ and $C_2 = C_2(p,s_0)$ such that 
\begin{align*}
&\int_{s_0}^{t} \lp{p}{\widetilde{u}_s(\cdot,s)}^{p}\,ds 
+ \int_{s_0}^{t} \lp{p}{A\widetilde{u}(\cdot,s)}^{p}\,ds \\
&\leq C_1\Bigl(\int_{s_0}^{t} \lp{p}{\widetilde{u}(\cdot,s) 
+ e^{s}(\gamma n_1(\cdot,s) + \delta n_2(\cdot,s))\nabla \phi}^{p}\,ds  
+1 \Bigr)
\\
&\leq C_2\Bigl(\int_{s_0}^{t} e^{ps}\lp{p}{u(\cdot,s)}^p\,ds 
+ \int_{s_0}^{t} e^{ps}(\lp{p}{n_1(\cdot,s)}^p + \lp{p}{n_2(\cdot,s)}^p)\,ds  
+1 
\Bigr)
\end{align*}
for all $t\in (s_0,\tmax)$. 
Hence we can prove this lemma. 
\end{proof}

\section{Boundedness. Proof of Theorem \ref{maintheorem1}}

We will prove 
Theorem \ref{maintheorem1} 
by preparing a series of lemmas in this section.


\begin{lem}\label{pote1}
There exists a constant $C > 0$ such that 
for $i = 1, 2$, 
     \[
     \int_{\Omega} n_i(\cdot, t) \leq C \quad 
    \mbox{for all}\ t \in (0, \tmax)
     \]
and 
    \begin{equation*}
    \int_{t}^{t+\tau}\int_{\Omega}n_{i}^2 \le C \quad 
    \mbox{for all}\ t \in (0, \tmax-\tau), 
    \end{equation*}
where 
    $\tau := \min\{1, \frac{1}{2}\tmax\}$.
\end{lem}
\begin{proof}
The above lemma can be proved 
by the same argument as in the proof of \cite[Lemma 3.1]{HKMY_1}.
\end{proof}
 \begin{lem}\label{lem;grad;c}
There exists a positive constant $C$ such that 
     $$
     \lp{2}{\nabla c(\cdot, t)} \leq C
     \quad 
     \mbox{for all}\ t\in (0,\tmax).
     $$
Moreover, for all $p \in [1, 6)$ 
there exists a positive constant $C(p)$ 
such that 
     $$
     \lp{p}{c(\cdot, t)} \leq C(p) 
     \quad 
     \mbox{for all}\ t\in (0,\tmax).
     $$
 \end{lem}
\begin{proof}
Integrating the third equation in \eqref{P} 
over $\Omega$ 
together with the $L^1$-estimates 
for $n_1$ and $n_2$ 
provided by Lemma \ref{pote1} 
implies that there exists $C_{1}>0$ such that 
\begin{align}\label{esti;l1;c}
\lp{1}{c(\cdot,t)} \le C_{1} 
\end{align}
for all $t\in(0,\tmax)$. 
We next see from 
an argument similar to that in the proofs of 
\cite[Lemmas 2.5 and 2.6]{Tao-Winkler_2015_3D} that 
there is $C_{2}>0$ such that 
\begin{align}\label{esti;l2;nablac}
\lp{2}{\nabla c(\cdot,t)} \le C_{2} 
\quad \mbox{for all}\ t>0. 
\end{align} 
Thanks to \eqref{esti;l1;c} and 
\eqref{esti;l2;nablac}, 
we have from the Gagliardo--Nirenberg inequality that 
\[
\lp{p}{c(\cdot, t)} 
\leq C_3\lp{2}{\nabla c(\cdot,t)}^{a}\lp{1}{c(\cdot,t)}^{1-a} 
+ C_3\lp{1}{c(\cdot, t)}
\le C_3C_2^aC_1^{1-a}+ C_3C_1
\]
for all $t\in (0,\tmax)$ 
with some positive constant $C_3$, 
where $a = \frac{6}{5}(1-\frac{1}{p}) \in (0, 1)$. 
Thus we can prove this lemma. 
\end{proof}


\begin{lem}\label{pote3}
For $r \in (1, 3)$ there exists a constant 
$C=C(r) > 0$ such that 
     $$
     \|u(\cdot, t)\|_{L^r(\Omega)} \leq C
     \quad 
     \mbox{for all}\ t\in (0,\tmax).
     $$
\end{lem}
\begin{proof} 
The $L^r$-boundedness of $u$ for $r\in (1,3)$ can be obtained 
from the well-known Neumann heat semigroup estimates together with 
Lemma \ref{pote1} (for more details, see \cite[Corollary 3.4]{W-2015}). 
\end{proof}

Now we fix $s_0 \in (0, \tmax) \cap (0, 1]$.
We will obtain the $L^p$-estimates 
for $n_1$ and $n_2$ 
by preparing a series of lemmas. 

\begin{lem}\label{pote4}
For all $p > 1$, $\ep > 0$ and $\ell > 0$ 
there exists a constant $C=C(p) > 0$ 
such that
 \begin{align*}
     \frac{1}{p}\int_{\Omega}n_1^p 
     + \frac{1}{p}\int_{\Omega}n_2^p 
     &\leq - (\mu - \ep - \ell)e^{-(p+1)t}\int_{s_0}^{t}e^{(p+1)s}
     \Bigl(\int_{\Omega}n_1^{p+1} + \int_{\Omega}n_2^{p+1} \Bigr) \\ 
              &\quad\,+ C\ell^{-p}\chi^{p+1}e^{-(p+1)t}
                               \int_{s_0}^{t}e^{(p+1)s}\int_{\Omega}|\Delta c|^{p+1} 
                      + C
     \end{align*}
on $(s_0, \tmax)$, 
where $\mu = \min\{\mu_1,\mu_2\}$, 
$\chi = \max\{\chi_1, \chi_2\}$.
\end{lem}
\begin{proof}
We can prove this lemma by using the same argument 
as in the proof of \cite[Lemma 3.4]{CKM}. 
\end{proof}

\begin{lem}\label{pote5}
For all $p \in (\frac{3}{2}, 2)$ there exists a constant 
$C=C(p) > 0$ such that 
 \begin{align*}
     \int_{s_0}^{t}e^{(p+1)s}\int_{\Omega}|\Delta c|^{p+1} 
      &\leq C\int_{s_0}^{t}e^{(p+1)s}
     \left(\into n_1^{p+1} + 
     \into n_2^{p+1}+\into |Au|^{p+1} \right) 
      + Ce^{(p+1)t} + C
     \end{align*}
for all $t \in (s_0, \tmax)$.
\end{lem}
\begin{proof}
Fix $\theta \in (\frac{3}{2}, 2)$ and 
put $\theta' = \frac{\theta}{\theta - 1} \in (2, 3)$.
From Lemma \ref{pre2} we have  
      \begin{align}\label{pote1.8}
      \int_{s_0}^{t}e^{(p+1)s}\int_{\Omega}|\Delta c|^{p+1} 
      &\leq C_{1}\int_{s_0}^{t}e^{(p+1)s}
      \int_{\Omega}
      |c + \alpha n_1 + \beta n_2 - u\cdot\nabla c|^{p+1} 
      \\ \notag 
          &\quad\,+ C_{1}e^{(p+1)s_0}
          \|c(\cdot, s_0)\|_{W^{2, p+1}(\Omega)}^{p+1}
      \end{align}
with some positive constant $C_{1}=C_{1}(p)$. 
It follows from Lemma \ref{lem;grad;c} and the H\"older inequality that 
there exists a positive constant $C_{2}=C_{2}(p)$ 
such that 
      \begin{align}\label{pote1.8tochu}
      &\int_{s_0}^{t}e^{(p+1)s}\int_{\Omega}
      |c + \alpha n_1 + \beta n_2 - u\cdot\nabla c|^{p+1} 
      \\ \notag 
      &\leq C_{2}\int_{s_0}^{t}e^{(p+1)s}
      \int_{\Omega} 
          (n_1^{p+1} + n_2^{p+1} + |u\cdot\nabla c|^{p+1}) 
          + C_{2}e^{(p+1)t} 
      \\ \notag 
      &\leq C_{2}\int_{s_0}^{t}e^{(p+1)s}
     \left(\into n_1^{p+1} + \into n_2^{p+1}\right) 
     \\ \notag
     &\quad\,
     + C_{2}\int_{s_0}^{t}e^{(p+1)s} 
     \|u(\cdot,s)\|_{L^{(p+1)\theta}(\Omega)}^{p+1}
     \|\nabla c(\cdot,s)\|_{L^{(p+1)\theta'}(\Omega)}^{p+1}\,ds 
+ C_{2}e^{(p+1)t}.
      \end{align}
Here the Gagliardo--Nirenberg inequality and Lemma \ref{lem;grad;c} 
imply  
      \begin{align}\label{pote1.9}
      \|\nabla c\|_{L^{(p+1)\theta'}(\Omega)}^{p+1} 
      &\leq C_{3}\|\Delta c\|_{L^{p+1}(\Omega)}^{a(p+1)}
                                          \|\nabla c\|_{L^2(\Omega)}^{(1-a)(p+1)} 
             + C_{3}\|\nabla c\|_{L^2(\Omega)}^{p+1} \\ \notag 
      &\leq C_{4}\|\Delta c\|_{L^{p+1}(\Omega)}^{a(p+1)} 
      + C_{4}
      \end{align} 
with some constants $C_{3}=C_{3}(p) > 0$ and 
$C_{4}=C_{4}(p) > 0$, 
where 
\[
 a := \frac{3-\frac{6}{(p+1)\theta'}}{5-\frac{6}{p+1}} \in (0, 1). 
 \]
We derive from 
\eqref{pote1.8}, \eqref{pote1.8tochu}, \eqref{pote1.9} and the Young inequality 
that there exists 
a positive constant $C_{5}=C_{5}(p)$ 
such that 
     \begin{align*}
     \int_{s_0}^{t}e^{(p+1)s}
     \int_{\Omega}|\Delta c|^{p+1} 
     &\leq C_{5}\int_{s_0}^{t}e^{(p+1)s}
     \left(\into n_1^{p+1} + \into n_2^{p+1}\right) 
     + a\int_{s_0}^{t}e^{(p+1)s}\into |\Delta c|^{p+1} 
\\\notag
     &\quad\,
     + C_{5}\int_{s_0}^{t}e^{(p+1)s} 
     \left(
     \|u(\cdot,s)\|_{L^{(p+1)\theta}(\Omega)}^{\frac{p+1}{1-a}}
     + \|u(\cdot,s)\|_{L^{(p+1)\theta}(\Omega)}^{p+1}
     \right)
     \,ds 
\\ \notag
     &\quad\,
          + C_{5}e^{(p+1)t}
          + C_{5}.
     \end{align*} 
Hence we can obtain that    
     \begin{align}\label{pote1.11}
     \int_{s_0}^{t}e^{(p+1)s}
     \int_{\Omega}|\Delta c|^{p+1} 
     &\leq 
     \frac{C_{5}}{1-a}\int_{s_0}^{t}e^{(p+1)s}
     \left(\into n_1^{p+1} + \into n_2^{p+1}\right) 
\\\notag
     &\quad\,
     + \frac{C_{5}}{1-a}\int_{s_0}^{t}e^{(p+1)s} 
     \left(
     \|u(\cdot,s)\|_{L^{(p+1)\theta}(\Omega)}^{\frac{p+1}{1-a}}
     + \|u(\cdot,s)\|_{L^{(p+1)\theta}(\Omega)}^{p+1}
     \right)
     \,ds 
\\ \notag
     &\quad\,
          + \frac{C_{5}}{1-a}e^{(p+1)t}
          + \frac{C_{5}}{1-a}.     
\end{align} 
Here we can choose $r \in (1, 3)$ such that 
     \begin{equation}\label{keypoint}
     \frac{5-\frac{6}{p+1}}{2-\frac{6}{(p+1)\theta}}\cdot
        \frac{\frac{1}{r} - \frac{1}{(p+1)\theta}}{\frac{2}{3} + \frac{1}{r} - \frac{1}{p+1}} 
     < 1
     \end{equation}
holds. 
Now it follows from 
the Gagliardo--Nirenberg inequality, Lemma \ref{pote3} 
and the Young inequality 
that there exists a constant $C_{6}=C_{6}(p)>0$ 
such that 
     \begin{align}\label{pote1.12}
     \|u(\cdot, s)\|_{L^{(p+1)\theta}(\Omega)}^{\frac{p+1}{1-a}} 
     &\leq \|Au(\cdot, s)\|_{L^{p+1}(\Omega)}^{\frac{p+1}{1-a}b}
     \|u(\cdot, s)\|_{L^r(\Omega)}^{\frac{p+1}{1-a}(1-b)} 
     \\\notag
     &\leq C_{6} + C_{6}\|Au(\cdot, s)\|_{L^{p+1}(\Omega)}^{p+1} 
     \end{align} 
with 
\[
  b := \frac{\frac{1}{r} - \frac{1}{(p+1)\theta}}{\frac{2}{3} + \frac{1}{r} 
              - \frac{1}{p+1}} 
\in (0, 1),\] 
because 
\[
  \frac{p+1}{1-a}b 
  = (p+1)\cdot
  \frac{5-\frac{6}{p+1}}{2-\frac{6}{(p+1)\theta}}\cdot
  \frac{\frac{1}{r} - \frac{1}{(p+1)\theta}}{\frac{2}{3} + \frac{1}{r} - \frac{1}{p+1}} 
  < p+1
\] 
holds from \eqref{keypoint}. 
We moreover have from 
the fact $\frac{1}{1-a}>1$, 
the Young inequality 
and \eqref{pote1.12} 
that
     \begin{equation}\label{pote1.12.1}
     \|u(\cdot, s)\|_{L^{(p+1)\theta}(\Omega)}^{p+1} 
     \leq C_{7} + C_{7}\|Au(\cdot, s)\|_{L^{p+1}(\Omega)}^{p+1}
     \end{equation}
for all $s \in (s_0, \tmax)$ with some constant $C_{7}=C_{7}(p)>0$. 
Thus, combining \eqref{pote1.11}, \eqref{pote1.12} and \eqref{pote1.12.1}, 
we see that 
there exists a constant $C_{8}=C_{8}(p) > 0$ 
such that 
\begin{align*}
     \int_{s_0}^{t}e^{(p+1)s}\int_{\Omega}|\Delta c|^{p+1} 
      &\leq C_8\int_{s_0}^{t}e^{(p+1)s}
     \left(\into n_1^{p+1} + 
     \into n_2^{p+1}\right) 
     \\&\quad\,
      +C_8\int_{s_0}^{t}e^{(p+1)s}\into |Au|^{p+1} 
      +C_8e^{(p+1)t} + C_8
     \end{align*}
for all $t \in (s_0, \tmax)$, 
which concludes the proof of this lemma. 
\end{proof}


We give the following lemma 
to control $\int_{s_0}^{t}e^{(p+1)s}\int_{\Omega}|\Delta c|^{p+1}$.
\begin{lem}\label{pote6}
For all $p \in (\frac{3}{2}, 2)$ there exists a constant 
$C=C(p) > 0$ such that 
 \begin{align*}
     \int_{s_0}^{t}e^{(p+1)s}\int_{\Omega}|Au|^{p+1}   
     &\leq 
     C\int_{s_0}^{t}e^{(p+1)s}\Bigl(\int_{\Omega}n_1^{p+1} + \int_{\Omega}n_2^{p+1} \Bigr)
     + Ce^{(p+1)t} +C
     \end{align*}
for all $t \in (s_0, \tmax)$.
\end{lem}
\begin{proof}
A combination of Lemmas \ref{po} and \ref{pote3} implies this lemma.
\end{proof}

The following lemma is concerned with 
the $L^p$-estimates 
for $n_1$ and $n_2$.
\begin{lem}\label{pote8}
For all $p \in (\frac{3}{2}, 2)$ there exists 
$\xi_0 > 0$ such that 
if $\frac{\chi}{\mu} < \xi_0$, then  
there exists a constant $C > 0$ such that 
     \[
     \|n_i(\cdot, t)\|_{L^p(\Omega)} \leq C 
     \quad \mbox{for all}\ t \in (0, \tmax)
     \ \mbox{and for}\ i=1, 2.
     \]
\end{lem}
\begin{proof} 
From Lemmas \ref{pote4}, \ref{pote5} and \ref{pote6} we have that 
there exists a constant $K(p) > 0$ such that 
\begin{align}\label{good}
     &\frac{1}{p}\int_{\Omega}n_1^p 
     + \frac{1}{p}\int_{\Omega}n_2^p 
\\ \notag&\leq 
- (\mu - \ep - \ell - K(p)\ell^{-p}\chi^{p+1})e^{-(p+1)t}\int_{s_0}^{t}e^{(p+1)s}
                                 \Bigl(\int_{\Omega}n_1^{p+1} + \int_{\Omega}n_2^{p+1} \Bigr) \\ \notag
              &\quad\,+ K(p)\ell^{-p}\chi^{p+1} + K(p)
     \end{align}
for all $t \in (s_0, \tmax)$, where $\ep>0$ and $\ell >0$.
Here there exists a constant $\xi_0=\xi_0(p) > 0$ such that 
\begin{align}\label{nice}
\inf_{\ell > 0} (\ell + K(p)\ell^{-p}\chi^{p+1}) = \frac{1}{\xi_0}\chi.
\end{align}
If $\frac{\chi}{\mu} < \xi_0$, 
then we see from \eqref{nice} that 
\[
\inf_{\ell > 0} (\ell + K(p)\ell^{-p}\chi^{p+1}) < \mu,
\]
and hence there exists a constant $\ell > 0$ such that 
\[
\mu > \ell + K(p)\ell^{-p}\chi^{p+1}. 
\]
Therefore, under the condition that $\frac{\chi}{\mu} < \xi_0$, 
we can find $\ep > 0$ satisfying 
\[
\mu-\ep -\ell-K(p)\ell^{-p}\chi^{p+1} \geq 0,
\]
which enables us to obtain from $\eqref{good}$ that 
\[
\frac{1}{p}\int_{\Omega}n_1^p 
+ \frac{1}{p}\int_{\Omega}n_2^p 
\leq K(p)\ell^{-p}\chi^{p+1} + K(p)
\]
holds on $(s_0, \tmax)$. 
\end{proof}

\begin{lem}\label{pote9}
Assume $\frac{\chi}{\mu} < \xi_0$. Then there 
exists a constant $C > 0$ such that 
     \[
     \|A^{\vartheta}u(\cdot, t)\|_{L^2(\Omega)} \leq C
    \quad \mbox{and}\quad  
    \lp{\infty}{u(\cdot,t)}\le C
    \]
for all $t\in (0,\tmax)$. 
\end{lem}
\begin{proof}
Thanks to Lemma \ref{pote8}, we can show this lemma 
by the same argument as in the proof of 
\cite[Lemma 3.9]{CKM}.
\end{proof}

 
\begin{lem}\label{pote12}
Assume $\frac{\chi}{\mu} < \xi_0$. Then there exist 
$r \in (3, 6) \cap (1, q]$ and $C=C(r) > 0$ such that 
     $$
     \|\nabla c(\cdot, t)\|_{L^r(\Omega)} \leq C
     \quad \mbox{for all}\ t \in (0, \tmax).
     $$
\end{lem}
\begin{proof} 
This proof is based on that of \cite[Lemma 3.10]{CKM}.
We first note from Lemma \ref{pote8} that 
for all $p\in(\frac{3}{2},2)$, 
\[
\|\alpha n_1(\cdot, s) + \beta n_2(\cdot, s)\|_{L^{p}(\Omega)} \leq C_1
\]
with some constant $C_1=C_1(p)>0$ 
and by choosing $\theta \in (r, 6)$ 
we see from Lemmas \ref{lem;grad;c} and \ref{pote9} that 
$$
\|u(\cdot, s)c(\cdot, s)\|_{L^{\theta}(\Omega)} \leq C_2
$$
with some constant $C_2>0$. 
Therefore an argument similar to that in 
the proof of 
\cite[Lemma 3.10]{CKM} implies this lemma. 
\end{proof}

\begin{lem}\label{pote13}
Assume $\frac{\chi}{\mu} < \xi_0$. 
Then there exists a constant $C > 0$ such that 
     \[
     \|n_i(\cdot, t)\|_{L^{\infty}(\Omega)} \leq C 
     \quad \mbox{for all}\ t \in (0, \tmax)
     \ \mbox{and for}\ i = 1, 2. 
     \]
\end{lem}
\begin{proof}
We can prove this lemma in 
the same way as in the proof of 
\cite[Lemma 3.11]{CKM}. 
\end{proof}


\begin{prth1.1}
Lemmas \ref{pre1}, \ref{pote9}, \ref{pote12} and \ref{pote13} directly 
drive Theorem \ref{maintheorem1}. 
\qed
\end{prth1.1}

\section{Asymptotic behavior. Proof of Theorem \ref{maintheorem2}}

We first recall the following lemma 
which will give stabilization in \eqref{P}.
\begin{lem}[{\cite[Lemma 4.6]{HKMY_1}}]\label{regularity}
Let $n \in C^0(\overline{\Omega}\times[0, \infty))$ 
satisfy that 
there exist constants $C > 0$ and $\alpha_0 \in (0, 1)$ 
such that  
      \begin{align*}
      \|n\|_{C^{\alpha_0, \frac{\alpha_0}{2}}(\overline{\Omega}\times[t, t+1])} 
      \leq C
      \quad \mbox{for all}\ t \geq 1.
      \end{align*}
Assume that 
\begin{align*}
\int_{1}^{\infty}\int_{\Omega} (n(\cdot, t) - N^*)^2 < \infty 
\end{align*}
with some constant $N^{*}$. 
Then 
\begin{align*}
n(\cdot, t) \to N^* \quad \mbox{in}\ L^{\infty}(\Omega)
\quad
\mbox{as}\ t \to \infty. 
\end{align*}
\end{lem}

\subsection{Case 1: $a_1, a_2 \in (0, 1)$} 

Now we assume that $\frac{\chi}{\mu} < \xi_0$ and 
will prove asymptotic behavior of solutions to \eqref{P} 
in the case $a_1, a_2 \in (0, 1)$. 
In this case we also suppose that 
there exists $\delta_1 > 0$ such that 
   \begin{align}\label{ass;case1}
     &4\delta_1 - (1+\delta_1)^2 a_1a_2 > 0 
     \quad \mbox{and}
   \\\label{ass;case1-2}
     &\frac{\chi_1^2(1-a_1)}{4a_1\mu_1(1-a_1a_2)}
     +
     \frac{\delta_1\chi_2^2(1-a_2)}{4a_2\mu_2(1-a_1a_2)}
     <   \frac{4\delta_1 - (1+\delta_1)^2a_1a_2}
    {a_1\alpha^2 \delta + a_2\beta^2 -a_1a_2\alpha\beta (1+\delta_1)}. 
   \end{align}
The following lemma asserts 
that the assumption of Lemma \ref{regularity} 
is satisfied in the case that $a_1,a_2<1$. 

\begin{lem}\label{lem;energy;case1}
Let $(n_1, n_2, c, u,P)$ be a solution to 
\eqref{P}. If $a_1, a_2 \in (0, 1)$, 
then there exist constants $C > 0$ and $\alpha_0 \in (0, 1)$ 
such that  
      \begin{align}\label{regnc}
      \|n_1\|_{C^{\alpha_0, \frac{\alpha_0}{2}}(\overline{\Omega}\times[t, t+1])}
      +\|n_2\|_{C^{\alpha_0, \frac{\alpha_0}{2}}(\overline{\Omega}\times[t, t+1])}
      +\|c\|_{C^{\alpha_0, \frac{\alpha_0}{2}}(\overline{\Omega}\times[t, t+1])}
      \leq C 
      \end{align}
for all $t \geq 1$. 
Moreover, $n_1$, $n_2$ and $c$ satisfy 
\begin{align}\label{infty1}
\int_{1}^{\infty} \int_\Omega (n_1-N_1)^2 
+ \int_{1}^{\infty} \int_\Omega (n_2-N_2)^2 
+ \int_{1}^{\infty} \int_\Omega (c-C^*)^2 
 < \infty, 
\end{align}
where 
       \[
       N_1 := \frac{1 - a_1}{1 - a_1a_2},\quad N_2 := \frac{1 - a_2}{1 - a_1a_2} 
       \quad C^* := \alpha N_1 + \beta N_2.
       \]
\end{lem}
\begin{proof}
We can first obtain from Lemmas \ref{pote9}, \ref{pote12}, \ref{pote13} and \cite{Ladyzenskaja} that \eqref{regnc} holds. 
Next we will confirm \eqref{infty1}. 
We put 
\begin{align*}
  E_1:=\int_\Omega 
       \left(
         n_1- N_1 - N_1\log \frac{n_1}{N_1}
       \right)
       + 
       \delta_1\frac{a_1\mu_1}{a_2\mu_2}\int_\Omega 
       \left(
         n_2- N_2 - N_2\log \frac{n_2}{N_2}
       \right)
       + \frac{\delta_2}{2}\int_\Omega (c-C^*)^2,
\end{align*}
where $\delta_1 > 0$ is a constant 
defined as 
in \eqref{ass;case1}--\eqref{ass;case1-2}
and 
$\delta_2 >0$ is a constant satisfying 
\begin{align*}
  \frac{a_2\mu_2\chi_1^2N_1}{4} 
  + \frac{\delta_1a_1\mu_1\chi_2^2N_2}{4}
  < 
  \delta_2 
  < \frac{a_1a_2\mu_1\mu_2 
    (4\delta_1 - (1+\delta_1)^2a_1a_2)}
    {a_1\alpha^2 \delta_1 + a_2\beta^2  
    -a_1a_2\alpha\beta (1+\delta_1)}. 
\end{align*}
Then noting from $\nabla\cdot u = 0$ that 
\[
  \int_{\Omega} u\cdot\nabla(\log n_i) = 0 
\] 
for all $i=1, 2$ 
and 
\[
\int_{\Omega} u\cdot\nabla (c^2) = 0,  
\] 
we see from an argument similar to that in the proof of 
\cite[Lemma 2.2]{Mizukami_improve} that 
there exists a constant $\ep_1>0$ such that 
\[
\frac d{dt}E_1(t) 
  \leq  
  - \ep_1 \Bigl(\int_\Omega (n_1-N_1)^2 + \int_\Omega (n_2-N_2)^2 
       + \int_\Omega (c-C^*)^2 \Bigr)  
\quad
  \mbox{for all}\ t>0.  
\]
Thus we have from the nonnegativity of $E_1$ that 
\[
\int_{1}^{\infty} \int_\Omega (n_1-N_1)^2 
+ \int_{1}^{\infty} \int_\Omega (n_2-N_2)^2 
+ \int_{1}^{\infty} \int_\Omega (c-C^*)^2 
\leq \frac{1}{\ep_1}E_1(1) < \infty, 
\]
which leads to \eqref{infty1}.
\end{proof}

\subsection{Case 2: $a_1 \geq 1 > a_2$} 

In this section 
we assume that $\frac{\chi}{\mu} < \xi_0$ and will obtain stabilization in \eqref{P} 
in the case $a_1 \geq 1 > a_2$. 
In this case 
we also suppose that 
there exist constants 
$\delta_1' > 0$ and $a_1' \in [1, a_1]$ such that 
\begin{align}\label{ass;case2}
&4\delta_1'-a_1'a_2(1+\delta_1')^2 > 0 
\quad \mbox{and}
\\\label{ass;case2-2}
&\mu_2 > 
\frac{\chi_2^2\delta_1'(\alpha^2 a_1'\delta_1' + \beta^2 a_2 
- \alpha\beta a_1'a_2(1+\delta_1'))}{4a_2(4\delta_1'-a_1'a_2(1+\delta_1')^2)}. 
\end{align}
We shall show the following lemma to verify that 
the assumption of Lemma \ref{regularity} 
is satisfied in the case that $a_1\ge 1> a_2$. 


\begin{lem}\label{lem;energy;case2}
Let $(n_1, n_2, c, u,P)$ be a solution to \eqref{P}. 
If $a_1 \geq 1> a_2$, then there exist constants $C > 0$ and $\alpha_0 \in (0, 1)$ 
such that  
      \begin{align}\label{regnc2}
      \|n_1\|_{C^{\alpha_0, \frac{\alpha_0}{2}}(\overline{\Omega}\times[t, t+1])}
      + \|n_2\|_{C^{\alpha_0, \frac{\alpha_0}{2}}(\overline{\Omega}\times[t, t+1])}
      + \|c\|_{C^{\alpha_0, \frac{\alpha_0}{2}}(\overline{\Omega}\times[t, t+1])} 
      \leq C 
      \end{align}
for all $t \geq 1$. 
Moreover, we have
\begin{align}\label{infty2}
\int_{1}^{\infty} \int_\Omega n_1^2 
+ \int_{1}^{\infty} \int_\Omega (n_2-1)^2 
+ \int_{1}^{\infty} \int_\Omega (c-\beta)^2 
 < \infty. 
\end{align}
\end{lem}
\begin{proof}
We first see from Lemmas \ref{pote9}, \ref{pote12}, \ref{pote13} and \cite{Ladyzenskaja} that \eqref{regnc2} holds. 
Next we will show \eqref{infty2}. 
We put 
\begin{align*}
  E_2:=\int_\Omega n_1
       + \delta_1'\frac{a_1'\mu_1}{a_2\mu_2}\int_\Omega (n_2- 1 - \log n_2)
       + \frac{\delta_2'}{2}\int_\Omega (c-\beta)^2,
\end{align*}
where $\delta_1' > 0$ and $a_1' \in [1, a_1]$ 
are constants defined 
as in \eqref{ass;case2}--\eqref{ass;case2-2} 
and $\delta_2' >0$ is a constant 
satisfying 
\begin{align*}
\frac{a_1'\mu_1\chi_2^2\delta_1'}{4a_2\mu_2} < \delta_2' < 
\frac{a_1'\mu_1(4\delta_1'-a_1'a_2(1+\delta_1')^2)}
                   {\alpha^2a_1'\delta_1' + \beta^2a_2-\alpha\beta a_1'a_2(1+\delta_1')}.
\end{align*}
Then noting from $\nabla\cdot u = 0$ that 
\[
\int_{\Omega} u\cdot\nabla(\log n_i) = 0 
\] 
for $i=1, 2$ 
and 
\[
\int_{\Omega} u\cdot\nabla (c^2) = 0,  
\] 
we derive from an argument similar to 
that in the proof of 
\cite[Lemma 4.1]{Mizukami_DCDSB} that 
there exists a constant $\ep_2>0$ such that 
\[
\frac d{dt}E_2(t) 
  \leq  
  - \ep_2 \Bigl(\int_\Omega n_1^2 + \int_\Omega (n_2 - 1)^2 
       + \int_\Omega (c-\beta)^2 \Bigr)  
\quad
  \mbox{for all}\ t>0.  
\]
Hence we obtain from the nonnegativity of $E_1$ that 
\[
\int_{1}^{\infty} \int_\Omega n_1^2 
+ \int_{1}^{\infty} \int_\Omega (n_2 - 1)^2 
+ \int_{1}^{\infty} \int_\Omega (c-\beta)^2 
\leq \frac{1}{\ep_2}E_2(1) < \infty, 
\]
which means that the desired estimate 
\eqref{infty2} holds. 
\end{proof}

\subsection{Convergence for $u$}

Finally we provide the following lemma with respect to 
the decay properties of $u$.   

\begin{lem}\label{lem;decaycu}
Under the assumptions of 
Theorems \ref{maintheorem1} and \ref{maintheorem2}, 
the solution of \eqref{P} 
has the following property\/{\rm :} 
  \begin{align*}
    \lp{\infty}{u(\cdot,t)}\to 0
    \quad \text{as}\ t\to\infty.
  \end{align*}
\end{lem}
\begin{proof}
From Lemmas \ref{lem;energy;case1} and \ref{lem;energy;case2} 
the same argument as in the proof of 
\cite[Lemma 4.6]{CKM} 
implies this lemma. 
\end{proof}

\subsection{Proof of Theorem \ref{maintheorem2}}

\begin{prth1.2}
A combination of Lemmas \ref{regularity}, 
\ref{lem;energy;case1}, \ref{lem;energy;case2} 
and \ref{lem;decaycu} directly leads to 
Theorem \ref{maintheorem2}. 
\qed
\end{prth1.2}

\section*{Acknowledgement}
M.M. is supported by JSPS Research 
Fellowships for Young Scientists (No. 17J00101). 
A major part of this work was written 
while S.K. and M.M. visited Universit\"at 
Paderborn under the support 
from Tokyo University of Science.


\end{document}